\documentclass{article}
\usepackage[utf8]{inputenc}
\usepackage{mathtools}
\usepackage{amsthm}
\usepackage{cite}
\usepackage{authblk}
\usepackage{dsfont}
\usepackage{comment}
\usepackage{epsfig}
\usepackage[T1]{fontenc}
\usepackage[utf8]{inputenc}
\usepackage{mathptmx}
\usepackage{amssymb}
\usepackage{blkarray}
\usepackage{amsmath,epsfig}
\usepackage{graphicx,graphics}
\usepackage{url}
\usepackage{hyperref}
\usepackage[capitalise]{cleveref}
\usepackage{textcomp}
\usepackage{enumitem}
\usepackage{comment}
\usepackage{pdflscape}
\usepackage{titlesec}
\usepackage{bm,bbm}
\usepackage{float,lscape}
\usepackage{subfig}

\usepackage{epstopdf}
\usepackage{multirow}
\usepackage{array}
\usepackage{booktabs}
\usepackage{graphicx}
\usepackage{caption}
\usepackage{lettrine}
\usepackage{enumitem}
\usepackage{lscape}
\usepackage{rotating}
\usepackage{booktabs}
\usepackage{longtable}
\usepackage{mathptmx}
\usepackage{anyfontsize}
\usepackage{t1enc}
\usepackage{relsize}
\usepackage{placeins}
\usepackage{nomencl}
\usepackage{rotating}
\usepackage{glossaries}
\usepackage{scrlayer}
\usepackage{enumitem}
\usepackage{comment}
\usepackage{amsthm}
\usepackage{amssymb}
\usepackage[titletoc]{appendix}
\usepackage[nottoc]{tocbibind}

\newtheorem{lemma}{Lemma}
\newtheorem{theorem}{Theorem}

\begin{document}
%\nocite{*}
\title{Improved Bounds for Multicovering Hypergraphs}
\author[1]{Anand Babu}
\author[2]{Sundar Vishwanathan}
\affil[1]{Department of Computer Science {\&} Engineering \\ NIT Calicut}
\affil[2]{Department of Computer Science {\&} Engineering \\ IIT Bombay}
%\author{}
\date{}
\maketitle
\begin{abstract}
The minimum number of bicliques needed to cover the edge set of the complete graph on $n$ vertices is $\lceil \log_2 n \rceil$. The Graham-Pollak theorem states that at least $n-1$ bicliques are required to partition the edge set of the complete graph on $n$ vertices. In this paper, we provide improvements for the generalizations of coverings of graphs and hypergraphs for some specific multiplicities. We also study an extension of the Katona-Szemer\'edi theorem to $r$-uniform hypergraphs.

\end{abstract}
\section{Introduction}
An $r$-uniform hypergraph $H$ (also referred to as an $r$-graph) is said to be $r$-partite if its vertex set $V(H)$ can be partitioned into sets $V_1, V_2, \cdots, V_r$, so that every edge in the edge set $E(H)$ of $H$ intersects $V_i$ in one vertex. An $r$-partite cover of an  $r$-uniform hypergraph $H(V, E)$ is a collection of complete $r$-partite $r$-graphs such that each edge in the edge set $E$ is contained in some complete $r$-partite $r$-graphs. An $r$-partite partition of an $r$-uniform hypergraph $H(V, E)$ is a pair-wise disjoint collection of complete $r$-partite $r$-graphs such that each edge in the edge set $E$ is present in some complete $r$-partite $r$-graph. The minimum size of the collection of complete $r$-partite $r$-graphs that partitions the edge set of an $r$-uniform hypergraph $H$ is represented by $f_r(H)$. The complete $r$-uniform hypergraph with $n$ vertices has an edge set consisting of all $r$-sized subsets of $[n]$. For a complete $r$-uniform hypergraph on $n$ vertices, the minimum size of the collection of complete $r$-partite $r$-graphs that partition their edge set is denoted by $f_r(n)$.
\paragraph{}
The problem of determining $f_r(n)$ for $r>2$, was proposed by Aharoni and Linial \cite{alon1986decomposition}. For $r=2$, $f_2(n)$ is the minimum number of biclique subgraphs required to partition the edge set of the complete graph on $n$ vertices. Graham and Pollak(\hspace{1sp}\cite{graham1971addressing,graham1972embedding} see also \cite{babai1988linear} and \cite{graham1978distance}) proved that at least $n-1$ bicliques are required to partition the edge set of the complete graph $K_n$. Since the edges of the complete graph $K_n$ can be partitioned into $n-1$ disjoint bicliques, this shows that $f_2(n)=n-1$. The original proof by Graham and Pollak uses Sylvester’s law of inertia \cite{graham1972embedding}. Other proofs of the same were found by Tverberg \cite{tverberg1982decomposition}, Peck \cite{peck1984new}, and Vishwanathan \cite{vishwanathan2008polynomial} using linear algebraic methods. A combinatorial proof was given by
Vishwanathan \cite{vishwanathan2013counting}. The generalizations of the Graham Pollak theorem were provided by Alon \cite{alon1986decomposition} and showed that $f_3(n)=n-2$ and $f_r(n)=\theta(n^{\lfloor \frac{r}{2} \rfloor})$. Cioab{\u{a}}, K{\"u}ndgen and Verstra{\"e}te  \cite{cioabua2009decompositions} and later Cioab{\u{a}} and Tait \cite{cioabua2013variations} provided improvements in the lower order terms of $f_r(n)$. Later Leader, Mili{\'c}evi{\'c} and Tan \cite{leader2017decomposing} provided improved constructive bounds on $f_r(n)$. Further improvements on $f_r(n)$ are provided in \cite{leader2018improved}  \cite{babu2019bounds}. 
\paragraph{}
Alon \cite{alon1997neighborly} provided a one-to-one correspondence between $p$-neighborly family of standard boxes in $\mathbb{R}^d$ and the bipartite covering of a complete graph of cardinality $d$ such that each edge in the edge set of the complete graph is contained in at least one biclique and at most $p$ bicliques. Alon also provided bounds for the minimum number of bicliques required to cover the edges of a complete graph such that the edges are covered at least once and at most $p$ times. Multicovering the edge set in complete $r$-uniform hypergraphs is discussed in \cite{babu2021multicovering}.
\paragraph{}
The order of a hypergraph $H(V, E)$ is the number of vertices present in the hypergraph $H$. For an $r$-uniform hypergraph $H(V, E)$, the minimum of the sum of the orders of the complete $r$-partite $r$-graphs in a collection over all the collections of complete $r$-partite $r$-graphs that cover the edge set $E$ is denoted by $b_r(H)$. A classical theorem relating the biclique partition and the sum of the orders of the biclique cover is due to Katona and Szemer\'edi \cite{katona1967problem}. The Katona-Szemer\'edi theorem states that the minimum over the sum of the orders of a collection of bicliques that cover the edge set of a complete graph on $n$ vertices is $n\log n$. The generalization of the theorem for graphs with chromatic number $\chi$ was provided by Mubayi and Vishwanathan \cite{mubayi2009bipartite}.
\paragraph{}
This paper presents an extension of the Katona-Szemer\'edi theorem and some generalizations of the Graham-Pollak theorem. In Section $2$, we provide improved bounds for the minimum number of bicliques required to cover the edges of a complete graph such that the number of times the edges get covered is contained in some specific list of positive integers. Similar results on the minimum number of complete $3$-partite $3$-graphs required to cover the edge set of a complete $3$-uniform hypergraph such that the number of occurrences of the edges belongs to a specific list are also provided. In Section $3$, we provide an extension of the Katona-Szemer\'edi theorem to $r$-uniform hypergraphs by providing a lower bound for the sum of the orders of the collection of complete $r$-partite $r$-graphs that cover the edge set of an $r$-uniform hypergraph in terms of the chromatic number of the $r$-uniform hypergraph.

\section{Multicovering graphs}
\paragraph{}
Let $L=\{l_1,l_2,\cdots,l_i\}$ where $l_1,l_2,\cdots,l_i$ are positive integers. An \emph{L-biclique covering} of a graph $G$ is a collection of bicliques such that every edge of $G$ is contained in $l_i$ of the bicliques for $l_i \in L$. The $L$-biclique covering number, denoted by $bp_L(G)$, is the minimum number of bicliques in an $L$-biclique covering. This definition can be naturally extended to hypergraphs. Here, the definition is specified for complete $r$-uniform hypergraphs for specified lists only. Let $[p] = \{1, 2, \cdots, p\}$. An $r$-partite $p$-multicover of a complete $r$-uniform hypergraph $K_n^{r}$ is a collection of complete $r$-partite $r$-graphs such that every hyperedge of $K_n^{r}$ is contained in $t$ of the $r$-partite $r$-graphs for some $t \in [p]$. The minimum size of such a cover is called the $r$-partite $p$-multicovering number and is denoted by $f_r(n,p)$. Note that $bp_L(K_n)$ where the list $L=\{1,2,\cdots,p\}$ is the same as $f_2(n,p)$.
\paragraph{}
The problem of bipartite $p$-multicovering of the complete graph $K_n$ on $n$ vertices was first studied by Alon \cite{alon1997neighborly}. Alon proved that $(1+o(1))\bigg(\frac{p!}{2^p}\bigg)^{1/p}n^{1/p} \leq f_2(n,p) \leq (1+o(1))p n^{1/p}$. Though the bounds are asymptotically tight, there is still a constant gap between the bounds. Huang and Sudakov \cite{huang2012counterexample} improved the lower bound for $f_2(n,p)$ to $(1+o(1))\bigg(\frac{p!}{2^{p-1}}\bigg)^{1/p}n^{1/p} \leq f_2(n,p).$ Cioab{\u{a}} and Tait \cite{cioabua2013variations} provided a lower bound for $bp_L(G)$ for any list $L$ and graph $G$. They also provided constructive $L$-partite cover of $K_n$ for some specified lists, such as $L=\{1,2\},\{2,3\},\{1,2,4\},\{2,4,\cdots,2i\}$. Despite being asymptotically tight, not many lists are known for which the values are exactly known. It is part of folklore that for $L=\{1,2,\cdots,\lfloor \log n \rfloor$\}, $bp_L(K_n)$ is $\lceil \log_2 n\rceil$. Radhakrishnan, Sen, and Vishwanathan \cite{radhakrishnan2000depth} give another list for which the $bp_L(K_n)$ is exactly known. They show that $bp_L(K_n)=\frac{n}{2}$ for infinitely many values of even $n$ when $L$ is the list of odd numbers less than $n$. They also give similar results for $bp_L(K_n)$ a list of numbers congruent $1 (mod p)$ where $p$ is prime. For $L=\{\lambda\}$ for a constant $\lambda$, De Caen, Gregory, and Pritikin conjectured that $bp_L(K_n)=n-1$, for large $n$. This is known to be true for $\lambda \leq 18$. The bounds on $r$-partite $p$-multicovering of complete $r$-uniform hypergraphs are provided in \cite{babu2021multicovering}.
\subsection{Constructive upper bound for $bp_{\{2,3\}}(K_n)$}
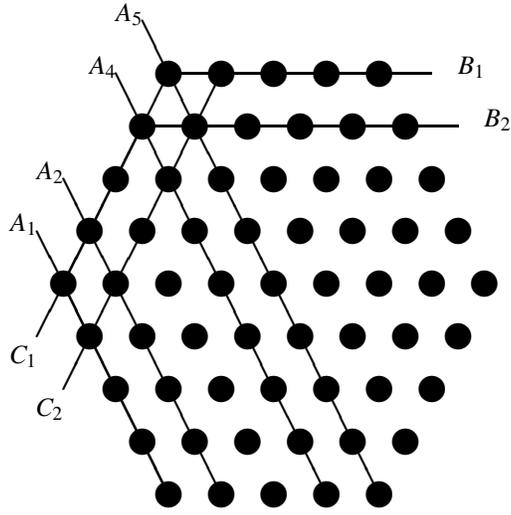
\begin{figure}[ht]
\centering
\begin{minipage}{5.5in}
\setlength{\unitlength}{0.7cm}
\begin{picture}(14,9)
\thicklines
\put(7,8){\circle*{0.5}}
\put(7,8){\line(1,0){5}}
\put(12.5,8){$B_1$}
\put(7,8){\line(-1,-2){2.5}}
\put(4,2.5){$C_1$}
\put(8,8){\line(-1,-2){3}}
\put(4.5,1.5){$C_2$}
\put(6.5,7){\line(1,0){6}}
\put(13,7){$B_2$}
\put(6.5,9){\line(1,-2){4.5}}
\put(6,9){$A_5$}
\put(7,0){\line(-1,2){2.5}}
\put(8,0){\line(-1,2){3}}
\put(4,5){$A_1$}
\put(4.5,6){$A_2$}
\put(8,8){\circle*{0.5}}
\put(9,8){\circle*{0.5}}
\put(10,8){\circle*{0.5}}
\put(11,8){\circle*{0.5}}
\put(6.5,7){\circle*{0.5}}
\put(6.5,7){\line(-1,-2){1.5}}
\put(6,8){\line(1,-2){4}}
\put(5.5,8){$A_4$}
\put(7.5,7){\circle*{0.5}}
\put(8.5,7){\circle*{0.5}}
\put(9.5,7){\circle*{0.5}}
\put(10.5,7){\circle*{0.5}}
\put(11.5,7){\circle*{0.5}}
\put(6,6){\circle*{0.5}}
\put(7,6){\circle*{0.5}}
\put(8,6){\circle*{0.5}}
\put(9,6){\circle*{0.5}}
\put(10,6){\circle*{0.5}}
\put(11,6){\circle*{0.5}}
\put(12,6){\circle*{0.5}}
\put(5.5,5){\circle*{0.5}}
\put(6.5,5){\circle*{0.5}}
\put(7.5,5){\circle*{0.5}}
\put(8.5,5){\circle*{0.5}}
\put(9.5,5){\circle*{0.5}}
\put(10.5,5){\circle*{0.5}}
\put(11.5,5){\circle*{0.5}}
\put(12.5,5){\circle*{0.5}}
\put(5,4){\circle*{0.5}}
\put(5,4){\line(1,-2){2}}
\put(6,4){\circle*{0.5}}
\put(7,4){\circle*{0.5}}
\put(8,4){\circle*{0.5}}
\put(9,4){\circle*{0.5}}
\put(10,4){\circle*{0.5}}
\put(11,4){\circle*{0.5}}
\put(12,4){\circle*{0.5}}
\put(13,4){\circle*{0.5}}
\put(5.5,3){\circle*{0.5}}
\put(6.5,3){\circle*{0.5}}
\put(7.5,3){\circle*{0.5}}
\put(8.5,3){\circle*{0.5}}
\put(9.5,3){\circle*{0.5}}
\put(10.5,3){\circle*{0.5}}
\put(11.5,3){\circle*{0.5}}
\put(12.5,3){\circle*{0.5}}
\put(6,2){\circle*{0.5}}
\put(7,2){\circle*{0.5}}
\put(8,2){\circle*{0.5}}
\put(9,2){\circle*{0.5}}
\put(10,2){\circle*{0.5}}
\put(11,2){\circle*{0.5}}
\put(12,2){\circle*{0.5}}
\put(6.5,1){\circle*{0.5}}
\put(7.5,1){\circle*{0.5}}
\put(8.5,1){\circle*{0.5}}
\put(9.5,1){\circle*{0.5}}
\put(10.5,1){\circle*{0.5}}
\put(11.5,1){\circle*{0.5}}
\put(7,0){\circle*{0.5}}
%\put(7,0){\line(1,0){3}}
\put(8,0){\circle*{0.5}}
\put(9,0){\circle*{0.5}}
\put(10,0){\circle*{0.5}}
\put(11,0){\circle*{0.5}}
\end{picture}
\end{minipage}
\caption{Hexagonal grid}
\end{figure}
\paragraph{}
In this subsection, we give an improved upper bound for $bp_{\{2,3\}}(K_n)$. This improves on the previous bound of $3\sqrt{2}\sqrt{n}$ by Cioab{\u{a}} and Tait \cite{cioabua2013variations}.
\paragraph{}
For $L=\{2,3\}$, order the $n$ vertices into a hexagonal grid with $m$ vertices on each side as in Figure 1. A hexagonal grid has three sides that are pairwise non-parallel. The rows parallel to each such side of the grid are denoted by $A_i$, $B_i$, and $C_i$ for $1 \leq i \leq 2m$ respectively. Consider the following collection of complete bipartite graphs with parts $A_i$ and $\cup_{j=i+1}^{2m} A_j$, the collection of complete bipartite graphs with parts $B_i$ and $\cup_{j=i+1}^{2m} B_j$ and the collection of complete bipartite graphs with parts $C_i$ and $\cup_{j=i+1}^{2m} C_j$. This forms a $\{2,3\}$ covering of $K_n$ since the edge whose vertices both lie in some row that is parallel to the sides of the hexagonal grid is covered exactly twice, and the rest of the vertices are covered thrice.
\paragraph{}
The total number of vertices in a hexagonal grid is given by
\begin{equation*}
    \begin{split}
    n&=2(m)+2(m+1)+\cdots+2(m+m-2)+(m+m-1)\\
    &=2[m(m-1)+(m-2)(m-1)/2]+(m+m-1)\\
    &=2m^2-2m+m^2-3m+2+2m-1\\
    &=3m^2-3m+1
    \end{split}
\end{equation*}
\paragraph{}
So for $n=3(m^2-m+1)$, $bp_{\{2,3\}}(K_n) \leq 6m-3 < 2 \sqrt{3} \sqrt{n}$.
\subsection{Constructive Upper bound for $f_3(n,4)$}
%\vspace{10mm}
\begin{figure}[ht]
\centering
\begin{minipage}{5 in}
\setlength{\unitlength}{0.8cm}
\begin{picture}(10,9)
\thicklines
\put(4,9){$R_1$}
\put(6,11){$C_1$}
\put(11,8){$M_1$}
\put(11,7){$M_2$}
\put(8.5,3.5){$N_1$}
\put(7.5,3.5){$N_2$}
\put(11,6){\line(-1,-1){2}}
\put(10,6){\line(-1,-1){2}}
\put(5,9){\line(1,0){5}}
\put(6,10){\line(0,-1){5}}
\put(9,10){\line(1,-1){2}}
\put(8,10){\line(1,-1){3}}
\put(6,5){\circle*{0.5}}
\put(7,5){\circle*{0.5}}
\put(8,5){\circle*{0.5}}
\put(9,5){\circle*{0.5}}
\put(10,5){\circle*{0.5}}
\put(6,6){\circle*{0.5}}
\put(7,6){\circle*{0.5}}
\put(8,6){\circle*{0.5}}
\put(9,6){\circle*{0.5}}
\put(10,6){\circle*{0.5}}
\put(6,7){\circle*{0.5}}
\put(7,7){\circle*{0.5}}
\put(8,7){\circle*{0.5}}
\put(9,7){\circle*{0.5}}
\put(10,7){\circle*{0.5}}
\put(6,8){\circle*{0.5}}
\put(7,8){\circle*{0.5}}
\put(8,8){\circle*{0.5}}
\put(9,8){\circle*{0.5}}
\put(10,8){\circle*{0.5}}
\put(6,9){\circle*{0.5}}
\put(7,9){\circle*{0.5}}
\put(8,9){\circle*{0.5}}
\put(9,9){\circle*{0.5}}
\put(10,9){\circle*{0.5}}
\end{picture}
\end{minipage}
\caption{Square grid}
\end{figure}
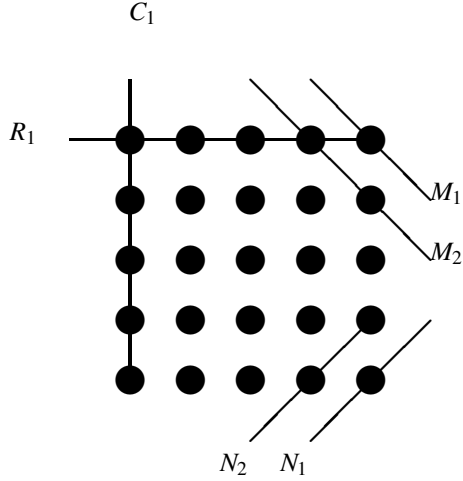

\paragraph{}
In this subsection, we give a constructive bound for $f_3(n,4)$. This bound is better than the general bound described in \cite{babu2021multicovering}.
\paragraph{}
For $L=\{1,2,3,4\}$ order the $n$ vertices into a square grid with $m$ rows and $m$ columns as in Figure 2. The rows are denoted by $R_i$ for $1 \leq i \leq m$ and the columns are denoted by $C_i$ for $1 \leq i \leq m$. The \emph{diagonal rows} are the lines that are parallel to the main diagonal of the square grid. The \emph{counter diagonal rows} are the lines that are parallel to the counter diagonal. The diagonal rows are denoted by $M_i$ and the counter diagonal rows are denoted by $N_i$ for $1 \leq i \leq 2m-1$. Some of the different types of rows in the square grid are depicted in Figure 2. Consider the following collection of complete $3$-partite $3$-graphs.
\begin{enumerate}
    \item Complete $3$-partite $3$-graphs with parts $R_i$, $\cup_{j=i+1}^m R_j$ and $\cup_{k=1}^{i-1} R_k$ for $2 \leq i \leq m-1$.
    \item Complete $3$-partite $3$-graphs with parts
    $C_i$, $\cup_{j=i+1}^m C_j$ and $\cup_{k=1}^{i-1} C_k$ for $2 \leq i \leq m-1$.
    \item Complete $3$-partite $3$-graphs with parts $M_i$, $\cup_{j=i+1}^m M_j$ and $\cup_{k=1}^{i-1} M_k$ for $2 \leq i \leq 2m-2$.
    \item Complete $3$-partite $3$-graphs with parts
    $N_i$, $\cup_{j=i+1}^m N_j$ and $\cup_{k=1}^{i-1} N_k$ for $2 \leq i \leq 2m-2$.
\end{enumerate}
\paragraph{}
Observe that for any three vertices in the square grid, there is always a row or a column or a diagonal row or a counter-diagonal row that passes through one of these vertices such that the other two vertices are on either side of it. The number of occurrences of an edge in the collection of $3$-partite $3$-graphs depends upon the number of such lines that pass through one of the vertices such that the other two vertices lie on either side of it. 
\paragraph{}
There are $n=m^2$ vertices in the square grid. The number of complete $3$-partite $3$-graphs is given by $$m-2 + m-2 + 2m-3 + 2m-3 = 6m-10.$$
\paragraph{}
So for $n=m^2$, $f_3(n,4)\leq 6m-10 \leq 6\sqrt{n}$. This shows that the general upper bound in \cite{babu2021multicovering} is not tight.

\section{Cover Order and the Chromatic Number}
\paragraph{}
In this section, we provide the extension of the generalization of the Katona-Szemer\'edi theorem by Mubayi and Vishwanathan \cite{mubayi2009bipartite} to $r$-uniform hypergraphs.
\paragraph{}
A hypergraph vertex coloring assigns $s$ colors to the vertices of the hypergraph. Such a coloring of a hypergraph $H$ is said to be a proper coloring if every edge in the edge set of the hypergraph $H$ contains at least two vertices with distinct colors. A hypergraph $H$ is said to be $k$-colorable if there exists a proper coloring of the hypergraph $H$ using at most $k$ distinct colors. The chromatic number of a hypergraph $H$, denoted by $\chi(H)$ is the smallest $k$ for which the hypergraph $H$ is $k$-colorable.
\paragraph{}
The minimum number of bipartite graphs required to cover the edge set of any graph $H$ with chromatic number $\chi(H)$ is $\lceil \log \chi(H) \rceil$ \cite{harary1977biparticity}. This result is part of folklore. One of the classical theorems that study an associated problem is the Katona-Szemer\'edi theorem \cite{katona1967problem}. The Katona-Szemer\'edi theorem \cite{katona1967problem} states that the minimum of the sum of the orders of a collection of complete bipartite graphs that cover the edge set of a complete graph on $n$ vertices is $n \log n$. Hansel \cite{hansel1964nombre}  provides an alternate proof for the Katona-Szemer\'edi theorem. Mubayi and Vishwanathan \cite{mubayi2009bipartite} provide a generalization of the theorem by showing that the sum of the orders of any collection of complete bipartite graphs that cover the edge set of $G$ is at least $k \log k - k \log \log k - k \log \log \log k$ for any graph $G$ with chromatic number $k$. This theorem is extended to $r$-uniform hypergraphs by using Jensen's inequality. Jensen's inequality states that,

\begin{lemma} \cite{jensen1906fonctions}
Let ${f}$ be a convex function of one real variable. Let $x_1,\dots,x_n\in \mathbb{R}$ and let $c_1,\dots, c_n\geq 0$ satisfy $c_1+\dots+c_n=1$. Then
$f(c_1x_1+\cdots+c_nx_n) \leq c_1f(x_1)+\cdots+c_nf(x_n)$.
\end{lemma}

\begin{lemma}\label{lemma KS}
Let $H$ be an $r$-uniform hypergraph with $n$ vertices with independence number $\alpha$. Then the sum of the orders of any collection of complete $r$-partite $r$-graphs that cover the edge set of $H$, $b_r(H) \geq (r-1) n \ln (\frac{n}{\alpha})$.
\end{lemma}
\begin{proof}
Consider a collection of complete $r$-partite $r$-graphs that cover the edge set of the $r$-uniform hypergraph. Let the vertex $i$ be present in $a_i$ complete $r$-partite $r$-graphs in the collection. Uniformly at random, pick one part each from every complete $r$-partite $r$-graph and remove all the vertices present in the part. The probability that a vertex $i$ is not removed from any of the complete $r$-parite $r$-graphs is $(\frac{r-1}{r})^{a_i}$. Hence, the expected size of the resulting subset of vertices is $\sum_{i=1}^n (1-\frac{1}{r})^{a_i}$. Note that this is an independent set. Since the independence number is $\alpha$, we have

\begin{equation}\label{eq:4.1}
    \begin{split}
        &\sum_{i=1}^n \bigg[1-\frac{1}{r}\bigg]^{a_i} \leq \alpha \\
    \end{split}
\end{equation}
$F(x)=(1-\frac{1}{r})^{x}$ is a convex function for fixed $r$, since, $F'(x)=-(1+\frac{1}{r-1})^{-x}\ln(1+\frac{1}{r-1})$ and $F''(x)=(1+\frac{1}{r-1})^{-x}\ln^2(1+\frac{1}{r-1})) > 0$. Applying Jensen's Inequality with $c_1=c_2=\cdots=c_n=\frac{1}{n}$ and since $\sum_{i=1}^n a_i=b_r(H)$ we have,
\begin{equation}\label{eq:4.2}
    \begin{split}
        \frac{1}{[1+\frac{1}{r-1}]^\frac{\sum_{i=1}^n a_i}{n}} &\leq \sum_{i=1}^n \frac{1}{n\cdot [1+\frac{1}{r-1}]^{a_i}}\\
        \frac{n}{[1+\frac{1}{r-1}]^\frac{\sum_{i=1}^n a_i}{n}} &\leq \sum_{i=1}^n \frac{1}{[1+\frac{1}{r-1}]^{a_i}}\\
         n \bigg[1-\frac{1}{r}\bigg]^{\frac{b_r(H)}{n}} &\leq \sum_{i=1}^n \frac{1}{[1+\frac{1}{r-1}]^{a_i}}=\sum_{i=1}^n \bigg[1-\frac{1}{r}\bigg]^{a_i}\\
    \end{split}
\end{equation}
From equations (\ref{eq:4.1}) and (\ref{eq:4.2}) we have
\begin{equation*}
    \begin{split}
        n \bigg[1-\frac{1}{r}\bigg]^{\frac{b_r(H)}{n}} &\leq \alpha
    \end{split}
\end{equation*}
Taking natural logarithm and using $(1+\frac{1}{x})^x \leq e$ we have,
\begin{equation*}
    \begin{split}
    \ln n &\leq \ln \alpha + \frac{b_r(H)}{n}\ln (\frac{r}{r-1})\\
    n\ln n- n \ln \alpha &\leq b_r(H)\ln (1+\frac{1}{r-1})\\
    n\ln n- n \ln \alpha &\leq b_r(H)\cdot \frac{1}{r-1}\\
    b_r(H) &\geq (r-1)n\ln (\frac{n}{\alpha})\\
    \end{split}
\end{equation*}
\end{proof}
\paragraph{}
Equivalently,
$$\alpha \geq \frac{n}{e^{\frac{b_r(H)}{(r-1)n}}}$$

\begin{theorem}
Let $H=(V, E)$ be an $r$-uniform hypergraph on $n$ vertices with chromatic number $k$, where $k$ is sufficiently large. The sum of the orders of any collection of complete $r$-partite $r$-graphs that cover the edge set of $H$ is at least $(r-1)^2 k \ln k(1-o(1))$.\\
\end{theorem}
\begin{proof}
Let the chromatic number of $H$, $\chi(H)=k$. If $n$ itself is  larger than $(r-1)^2 k\ln k$, we are done. Hence we may assume that $n$  is less than $(r-1)^2 k \ln k$. Let $H=H_0$. Starting with $H_0$, repeatedly remove independent sets of size given by Lemma \ref{lemma KS}, as long as the number of vertices is at least $(r-1)k$. Let $H_i$ denote the $i$-th graph in the sequence, and let $H_t$ denote the last graph in the sequence. Let $|V(H_i)|=n_i$ and $(1-\frac{1}{\beta})$ denote the maximum rate at which $n_i$ fall, over all $i$. That is, $\beta=max_i\bigg[e^{\frac{b_r(H_i)}{(r-1)n_i}}\bigg]$. Let this maximum be achieved for $i=p$. From the definition, we see that $n_{i+1} \leq n_i-\frac{n_i}{e^{\frac{b_r(H_i)}{(r-1)n_i}}}=n_i\Bigg[1-\frac{1}{e^{\frac{b_r(H_i)}{(r-1)n_i}}}\Bigg]$. Hence $n_t \leq n\big(1-\frac{1}{\beta}\big)^t < n e^{-\frac{t}{\beta}}$. Using the fact that $n_t \geq (r-1)k$ and $n<(r-1)^2 k\ln k$, we obtain $(r-1)k < n e^{-\frac{t}{\beta}}$ that is $t < \beta \ln \big[\frac{n}{(r-1)k}\big] < \beta \ln[(r-1)\ln k]$.
\paragraph{}
We consider two cases. First, suppose that $t$ is at least $\frac{k}{\ln k}$. Then, from the above two inequalities, we have $\frac{k}{\ln k} \leq e^{\frac{b_r(H_p)}{(r-1)n_p}} \ln[(r-1)\ln k]$. Taking natural logarithm, we get
\begin{equation*}
    \begin{split}
        \ln k - \ln \ln k &\leq \frac{b_r(H_p)}{(r-1)n_p}+ \ln \ln[(r-1)\ln k]\\
        b_r(H_p) &\geq (r-1)n_p\bigg[\ln k - \ln \ln k - \ln \big\{\ln (r-1) + \ln \ln k\big\}\bigg]\\
    \end{split}
\end{equation*}
Using the facts that $n_p>(r-1)k$ we have,
\begin{equation*}
    \begin{split}
        b_r(H_p)\geq (r-1)^2 k \bigg[\ln k-\ln \ln k-\ln \big \{\ln (r-1) + \ln \ln k \big \}\bigg]
    \end{split}
\end{equation*}

We now consider the case that $t$ is less than $\frac{k}{\ln k}$. Let $    H'$ be the hypergraph obtained after removing an independent set from $H_t$ as mentioned in Lemma \ref{lemma KS}. By definition of $t$, we have $|V(H')|<(r-1)k$. Also $\chi(H') \geq k - \frac{k}{\ln k}$.

Consider an optimal coloring of the hypergraph $H'$. Order the vertices such that the vertices in color class $i$ precede the vertices in color class ($i+1$). Let $C$ denote a greedy coloring of the vertices of $H'$ considered in this order. Note that $C$ has at most one color class with strictly less than ($r-1$) vertices since we color $H'$ greedily.

\paragraph{}
Let $p$ be the number of vertices present in the set of ($r-1$)-sized color classes and $q$ be the rest of the elements. Counting the number of elements in $H'$ we have,

\begin{equation*}
    \begin{split}
        p+q<(r-1)k
    \end{split}
\end{equation*}

Counting the number of color classes we have,

\begin{equation*}
    \begin{split}
        1+\frac{p}{r-1}+\frac{q}{r}\geq \chi(H')\geq k - \frac{k}{\ln k}
    \end{split}
\end{equation*}
From the above two inequalities, we have
\begin{equation*}
    \begin{split}
        1+\frac{p}{r-1}+\frac{(r-1)k-p}{r}&\geq k- \frac{k}{\ln k}\\
        1+\frac{p}{r(r-1)}+\bigg(1-\frac{1}{r}\bigg)k&\geq k- \frac{k}{\ln k}\\
        1+\frac{p}{r(r-1)}&\geq \frac{k}{r}-\frac{k}{\ln k}\\
        r(r-1)+p&\geq (r-1)k-\frac{r(r-1)k}{\ln k}\\
        p &\geq (r-1)k \bigg[1-\frac{r}{\ln k}\bigg]-r(r-1)\\
    \end{split}
\end{equation*}

Consider the sub-hypergraph $H''$ spanned by the vertices of the ($r-1$)-sized color classes in $H'$. Note that the color classes are ordered. Let $a_{i,j}$ be the $i$-th vertex in the $j$-th ($r-1$)-sized color class. Note that each vertex $a_{i,j}$ forms an edge with the $k$-th ($r-1$) sized color class for all $k$ where $k<j$ since coloring $C$ is a greedy coloring. Note that $H''$ has $\frac{p}{r-1}$ color classes in the optimal coloring. Suppose $H''$ has an independent set of size greater than $2(r-1)$, then color that independent set with one color and color the rest of the vertices with strictly less than $\frac{p}{r-1}-2$ colors, contradicting the assumption that $C$ is an optimal coloring. Therefore, the size of the independent set in $H''$ is at most $2(r-1)$. Applying lemma \ref{lemma KS} for $H''$, we have 
\begin{equation*}
    \begin{split}
        b_r(H'')&\geq (r-1)p\ln\bigg[\frac{p}{2(r-1)}\bigg]\\
    \end{split}
\end{equation*}
Since $p \geq (r-1)k \big[1-\frac{r}{\ln k}\big]-r(r-1)$ we have,
\begin{comment}
\begin{equation*}
    \begin{split}
        (r-1)p&\geq (r-1)^2 k-\frac{r(r-1)^2k}{\log k} - r(r-1)^2\\
        p \log (\frac{p}{r-1})& \geq \big[(r-1)k[1-\frac{r}{\log k}]-r(r-1)\big]\cdot \log \big[ k[1-\frac{r}{\log k}]-r(r-1) \big]\\
    \end{split}
\end{equation*}
\end{comment}

\begin{equation*}
    \begin{split}
        p &\geq (r-1)k- r(r-1)\frac{k}{\ln k}-r(r-1)\\
        & \geq (r-1)k-r(r-1)\bigg[\frac{k}{\ln k}+1\bigg]\\
        &\geq (r-1)k-\frac{2r(r-1)k}{\ln k}\\
    \end{split}
\end{equation*}
Therefore,
\begin{equation*}
    \begin{split}
         b_r(H'')&\geq (r-1)p\ln\bigg[\frac{p}{2(r-1)}\bigg]\\
         &\geq (r-1)\bigg[(r-1)k-\frac{2r(r-1)k}{\ln k}\bigg]\ln\bigg[\frac{k}{2}-\frac{rk}{\ln k}\bigg]\\
         &\geq \bigg[(r-1)^2k-\frac{2r(r-1)^2k}{\ln k} \bigg]\ln \bigg[\frac{k}{2}-\frac{rk}{\ln k}\bigg]\\
         &\geq (r-1)^2k \ln \bigg[\frac{k}{2}-\frac{rk}{\ln k}\bigg]-\frac{2r(r-1)^2k}{\ln k}\ln\bigg[\frac{k}{2}-\frac{rk}{\ln k}\bigg]\\
         &\geq (r-1)^2k \ln \bigg[\frac{k}{2}-\frac{rk}{\ln k}\bigg]-\frac{2r(r-1)^2k}{\ln k}\ln\bigg[\frac{k}{2}\bigg]\\
         &\geq (r-1)^2k \ln \bigg[k\bigg(\frac{1}{2}-\frac{r}{\ln k}\bigg)\bigg]-2r(r-1)^2k\\
         &\geq (r-1)^2k \bigg(\ln k+ \ln \bigg(\frac{\ln k-2r}{2\ln k}\bigg)\bigg)-2r(r-1)^2k\\    
         &\geq (r-1)^2k \bigg(\ln k -\ln \ln k-1\bigg)-2r(r-1)^2k\\
         &\geq (r-1)^2k \bigg(\ln k -\ln \ln k-1 -2r\bigg)\\
    \end{split}
\end{equation*}
\end{proof}
\bibliographystyle{plain}
\bibliography{reference}

@article{leader2017decomposing,
title = "Decomposing the complete $r$-graph",
journal = "Journal of Combinatorial Theory, Series A",
volume = "154",
number = "Supplement C",
pages = "21 - 31",
year = "2018",
issn = "0097-3165",
doi = "https://doi.org/10.1016/j.jcta.2017.08.008",
url = "http://www.sciencedirect.com/science/article/pii/S0097316517301024",
author = {I. Leader and L.  Mili{\'c}evi{\'c} and T.S. Tan}
}

@article{leader2018improved,
  title={Improved Bounds for the {Graham}-{Pollak} {Problem} for {Hypergraphs}},
  author={Leader, I. and Tan, T.S.},
  journal={The Electronic Journal of Combinatorics},
  volume={25},
  number={1},
  pages={1--4},
  year={2018}
}

@article{alon1986decomposition,
  title={Decomposition of the complete $r$-graph into complete $r$-partite $r$-graphs},
  author={N. Alon},
  journal={Graphs and Combinatorics},
  volume={2},
  number={1},
  pages={95--100},
  year={1986},
  publisher={Springer}
}

@article{graham1971addressing,
  title={On the addressing problem for loop switching},
  author={R.L. Graham and H.O. Pollak},
  journal={Bell Labs Technical Journal},
  volume={50},
  number={8},
  pages={2495--2519},
  year={1971},
  publisher={Wiley Online Library}
}

@incollection{graham1972embedding,
  title={On embedding graphs in squashed cubes},
  author={R.L. Graham and H.O. Pollak},
  booktitle={Graph theory and applications},
  pages={99--110},
  year={1972},
  publisher={Springer}
}

@article{graham1978distance,
  title={Distance matrix polynomials of trees},
  author={R.L. Graham and L. Lov{\'a}sz},
  journal={Advances in Mathematics},
  volume={29},
  number={1},
  pages={60--88},
  year={1978},
  publisher={Elsevier}
}

@article{tverberg1982decomposition,
  title={On the decomposition of {$K_n$} into complete bipartite graphs},
  author={H. Tverberg},
  journal={Journal of Graph Theory},
  volume={6},
  number={4},
  pages={493--494},
  year={1982},
  publisher={Wiley Online Library}
}

@article{peck1984new,
  title={A new proof of a theorem of {Graham} and {Pollak}},
  author={G.W. Peck},
  journal={Discrete Mathematics},
  volume={49},
  number={3},
  pages={327--328},
  year={1984},
  publisher={Elsevier}
}

@article{vishwanathan2008polynomial,
  title={A polynomial space proof of the {Graham}--{Pollak} theorem},
  author={S. Vishwanathan},
  journal={Journal of Combinatorial Theory, Series A},
  volume={115},
  number={4},
  pages={674--676},
  year={2008},
  publisher={Elsevier}
}

@article{vishwanathan2013counting,
  title={A counting proof of the {Graham}--{Pollak} Theorem},
  author={S. Vishwanathan},
  journal={Discrete Mathematics},
  volume={313},
  number={6},
  pages={765--766},
  year={2013}
}

@article{cioabua2013variations,
  title={Variations on a theme of {Graham} and {Pollak}},
  author={S.M. Cioab{\u{a}} and M. Tait},
  journal={Discrete Mathematics},
  volume={313},
  number={5},
  pages={665--676},
  year={2013},
  publisher={Elsevier}
}

@article{cioabua2009decompositions,
  title={On decompositions of complete hypergraphs},
  author={S.M. Cioab{\u{a}} and A. K{\"u}ndgen and J. Verstra{\"e}te},
  journal={Journal of Combinatorial Theory, Series A},
  volume={116},
  number={7},
  pages={1232--1234},
  year={2009},
  publisher={Elsevier}
}

@article{huang2012counterexample,
  title={A counterexample to the {Alon}-{Saks}-{Seymour} conjecture and related problems},
  author={Huang, H. and Sudakov, B.},
  journal={Combinatorica},
  volume={32},
  number={2},
  pages={205--219},
  year={2012},
  publisher={Springer}
}

@incollection{alon1997neighborly,
  title={Neighborly families of boxes and bipartite coverings},
  author={Alon, N.},
  booktitle={The Mathematics of Paul Erd{\"o}s II},
  pages={27--31},
  year={1997},
  publisher={Springer}
}

@article{babu2019bounds,
  author    = {A. Babu and
               S. Vishwanathan},
  title     = {Bounds for the {G}raham-{P}ollak theorem for hypergraphs},
  journal   = {Discrete Mathematics},
  volume    = {342},
  number    = {11},
  pages     = {3177--3181},
  year      = {2019},
  publisher = {Elsevier}
}

@inproceedings{radhakrishnan2000depth,
  title={Depth-3 Arithmetic Circuits for ${S}_n^2(X)$ and {Extensions} of the {Graham}-{Pollak} Theorem},
  author={Radhakrishnan, J. and Sen, P. and Vishwanathan, S.},
  booktitle={International Conference on Foundations of Software Technology and Theoretical Computer Science},
  pages={176--187},
  year={2000},
  organization={Springer}
}

@book{babai1988linear,
  title={Linear algebra methods in combinatorics: with applications to geometry and computer science.},
  author={Babai, L. and Frankl, P.},
  year={1992},
  publisher={University of Chicago}
}

@article{harary1977biparticity,
  title={The biparticity of a graph},
  author={Harary, F. and Hsu, D. and Miller, Z.},
  journal={Journal of graph theory},
  volume={1},
  number={2},
  pages={131--133},
  year={1977},
  publisher={Wiley Online Library}
}

@article{katona1967problem,
  title={On a problem of graph theory},
  author={Katona, G. and Szemer{\'e}di, E.},
  journal={Studia Scientiarum Mathematicarum Hungarica},
  volume={2},
  pages={23--28},
  year={1967},
  publisher={Akad{\'e}miai Kiad{\'o}}
}

@article{hansel1964nombre,
  title={Nombre minimal de contacts de fermeture n{\'e}cessaires pour r{\'e}aliser une fonction bool{\'e}enne sym{\'e}trique de n variables},
  author={Hansel, G.},
  journal={Comptes Rendus Hebdomadaires Des Seances De L Academie Des Sciences},
  volume={258},
  number={25},
  pages={6037},
  year={1964},
  publisher={GAUTHIER-VILLARS/EDITIONS ELSEVIER 23 RUE LINOIS, 75015 PARIS, FRANCE}
}

@article{mubayi2009bipartite,
  title={Bipartite Coverings and the Chromatic Number},
  author={Mubayi, D. and Vishwanathan, S.},
  journal={the electronic journal of combinatorics},
  volume={16},
  number={1},
  pages={N34},
  year={2009}
}

@article{jensen1906fonctions,
  title={Sur les fonctions convexes et les in{\'e}galit{\'e}s entre les valeurs moyennes},
  author={Jensen, J. L. W. V.},
  journal={Acta mathematica},
  volume={30},
  number={1},
  pages={175--193},
  year={1906},
  publisher={Springer}
}

@article{babu2021multicovering,
  title={Multicovering hypergraphs},
  author={Babu, A. and Vishwanathan, S.},
  journal={Discrete Mathematics},
  volume={344},
  number={6},
  pages={112386},
  year={2021},
  publisher={Elsevier}
}
\end{document}